\documentclass[10pt,reqno]{amsart}
\usepackage[pdftex]{graphicx}
\usepackage{amsmath, amssymb,amsthm}
\usepackage{enumerate}
\usepackage{bbm} 
\usepackage{tikz}
\usepackage{hyperref}  
\hypersetup{
	colorlinks,
	citecolor=blue,
	filecolor=blue,
	linkcolor=blue,
	urlcolor=blue
}

\newtheorem{thm}{Theorem}[section]
\newtheorem{prop}[thm]{Proposition}
\newtheorem{lem}[thm]{Lemma}
\newtheorem{cor}[thm]{Corollary}
\theoremstyle{definition}
\newtheorem{defn}[thm]{Definition}
\newtheorem{rmk}[thm]{Remark}
\newtheorem{ex}[thm]{Example}

\newcommand{\comments}[1]{}

\newcommand{\Z}{\mathbb{Z}}
\newcommand{\R}{\mathbb{R}}

\newcommand{\Sb}{\mathbb{S}}

\renewcommand{\Im}{\operatorname{\mathrm{Im}}}

\renewcommand{\epsilon}{\varepsilon}

\numberwithin{equation}{section}

\begin{document}

\title[Nonhomogeneous delocalization transition]{A 
localization-delocalization transition for nonhomogeneous random
matrices}

\author{Laura Shou}
\address{School of Mathematics, University of Minnesota, 206 Church St 
SE, Minneapolis, MN 55455, USA}
\email{shou0039@umn.edu}

\author{Ramon van Handel}
\address{Fine Hall 207, Princeton University, Princeton, NJ 08544, USA}
\email{rvan@math.princeton.edu}

\begin{abstract}
We consider $N\times N$ self-adjoint Gaussian random matrices defined by 
an arbitrary deterministic sparsity pattern with $d$ nonzero entries per 
row.  We show that such random matrices exhibit a canonical 
localization-delocalization transition near the edge of the spectrum: when 
$d\gg\log N$ the random matrix possesses a delocalized \emph{approximate} 
top eigenvector, while when $d\ll\log N$ any approximate top eigenvector 
is localized. The key feature of this phenomenon is that 
it is universal with respect to the sparsity pattern, in contrast to
the delocalization properties of \emph{exact} eigenvectors which are
sensitive to the specific sparsity pattern of the random matrix.
\end{abstract}


\maketitle

\section{Introduction}\label{sec:mat-intro}

Understanding the eigenvectors of large random matrices, particularly 
whether they are delocalized or localized, is of interest in many areas 
including mathematical physics, computer science, and combinatorics. A 
delocalized vector is one with roughly equal mass spread throughout its 
coordinates, while a localized vector has much of its mass concentrated on 
relatively few coordinates. The guiding example of delocalization arises 
in rotationally invariant ensembles such as the classical Gaussian 
orthogonal ensemble (GOE): their eigenvectors are uniformly distributed on 
the unit sphere, and are therefore always delocalized. Properties of 
uniform random vectors on the sphere can therefore be used as a benchmark 
for measuring delocalization. Much work in this direction has been done 
for general \emph{Wigner-type matrices} 
\cite{rmt-book-dynamical,eigenvectors-survey}, for a variety of 
indicators of delocalization.

On the other hand, the most localized vectors are simply the coordinate 
directions, with all mass concentrated on a single coordinate. These arise 
as eigenvectors of diagonal matrices, such as a diagonal matrix with 
i.i.d.\ diagonal Gaussian entries. To interpolate between the two extremes 
of a diagonal matrix and a Gaussian Wigner matrix, one can consider models 
of varying degrees of sparseness. One such model of interest in 
mathematical physics is random band matrices, which are zero outside of a 
band around the diagonal, and whose eigenvectors are conjectured to 
undergo a phase transition from localization to delocalization depending 
on the band width. See, e.g., \cite{bourgade-survey} for a survey of this topic.

In this paper we will consider Gaussian random matrices with an 
\emph{arbitrary} sparsity pattern, which is only assumed to be $d$-regular 
in the sense that there are $d$ nonzero entries in each row and column. 
This model includes the above mentioned Gaussian models (GOE, diagonal, 
and band matrices), as well as many other matrices that may be highly 
nonhomogeneous. As will be illustrated below using simple examples, the 
delocalization of eigenvectors is sensitive to the choice of sparsity 
pattern; an understanding of such questions for arbitrary sparsity 
patterns is far beyond current technology. In contrast, we will show that 
a much simpler phenomenon arises in this general setting near the edge of 
the spectrum: the \emph{approximate} top eigenvectors (by which we will 
mean unit vectors $v$ with $\|Xv\|_2$ close to $\|X\|$) exhibit a 
canonical localization-delocalization transition at $d\sim\log N$. This 
shows in particular that while the behavior of the \emph{exact} 
eigenvectors may be sensitive to the structure of the model, a weaker 
notion of (de)localization can nonetheless arise universally with respect 
to the sparsity pattern.

\subsection{Sparse Gaussian matrices}
\label{sec:smat}

The following general 
model will be considered throughout this paper. Fix an 
arbitrary $d$-regular undirected graph $G=([N],E)$ with $N$ vertices, 
which may contain self-loops but no multiple edges between the same pair 
of vertices. We now define an $N\times N$ self-adjoint random matrix $X_N$ 
by setting $(X_N)_{xy} = 1_{x\sim y}g_{xy}$, where $g_{xy}$ are 
independent standard (real) Gaussian variables modulo symmetry 
$g_{yx}=g_{xy}$, and $x\sim y$ denotes that $x$ and $y$ are connected by 
an edge in $G$. From now on, we fix a sequence of such matrices $X_N$
indexed by $N$; it is implicit in the notation that $d$ and $G$ 
depend on $N$.

The graph $G$ is used here merely as a convenient way to encode an 
arbitrary sparsity pattern of the entries of $X_N$. For example, the 
complete graph with $N$ vertices yields a Wigner matrix with i.i.d.\ 
standard Gaussian entries (modulo symmetry), while the graph consisting of 
$N$ isolated points with self-loops corresponds to the diagonal matrix 
with i.i.d.\ standard Gaussians on the diagonal. Two intermediate examples 
are periodic band matrices with band width $d$, which are generated by the 
graph on $\Z/N\Z$ with edges between nodes within distance $\frac{d-1}{2}$ 
of each other, and block Wigner matrices with block size $d$, which are 
generated by the disjoint union of $\frac{N}{d}$ complete graphs with $d$ 
vertices (cf.\ Figure~\ref{fig:mat}). Let us emphasize that while these 
simple examples possess many special symmetries that could potentially 
facilitate their analysis, a general $d$-regular graph can be highly 
nonhomogeneous and need not possess any tractable structure.
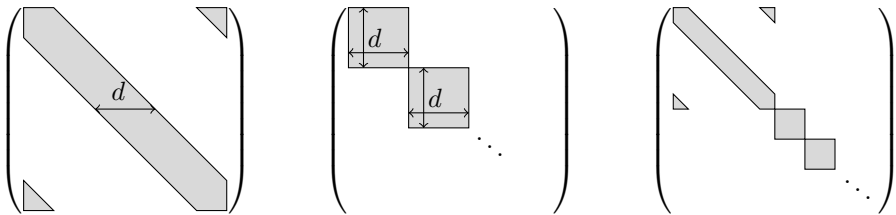
\begin{figure}
\centering
\begin{tikzpicture}
\def\psize{1.5cm}
\def\msize{1.35cm}
\def\w{.4cm} 
\draw[fill=gray!30] (-\msize,\msize-\w)--(-\msize,\msize)--(-\msize+\w,\msize)--(\msize,-\msize+\w)--(\msize,-\msize)--(\msize-\w,-\msize)--cycle;
\draw[xshift=\msize,yshift=\msize,fill=gray!30] (0,0)--(0,-\w)--(-\w,0)--cycle;
\draw[xshift=-\msize,yshift=-\msize,fill=gray!30] (0,0)--(0,\w)--(\w,0)--cycle;
\draw[<->] (-\w,0)--(\w,0);
\node[above] {$d\phantom{d}$};
\node at (-\psize,0) {$\left(\vphantom{\rule{\psize}{\psize}}\right.$};
\node at (\psize,0) {$\left.\vphantom{\rule{\psize}{\psize}}\right)$};
\end{tikzpicture}
\hspace{0.5cm}
\begin{tikzpicture}
\def\psize{1.5cm}
\def\msize{1.35cm}
\def\w{.8cm} 
\foreach \b in {0,1}{
	\draw[fill=gray!30,xshift=-\msize+\b*\w,yshift=\msize-\b*\w] (0,0)--(\w,0)--(\w,-\w)--(0,-\w)--cycle;
	\draw[xshift=-\msize+\b*\w,yshift=\msize-\b*\w-.2cm, <->] (0,-.4)--(\w,-.4);
	\draw[xshift=-\msize+\b*\w,yshift=\msize-\b*\w, <->] (.2,-\w)--(.2,0);
	\node at (-\msize+\b*\w+.35cm,\msize-\b*\w-.4cm) {$d$};
}
\node at (.2,0) [below right] {$\ddots$};
\node at (-\psize,0) {$\left(\vphantom{\rule{\psize}{\psize}}\right.$};
\node at (\psize,0) {$\left.\vphantom{\rule{\psize}{\psize}}\right)$};
\end{tikzpicture}
\hspace{0.5cm}
\begin{tikzpicture}
\def\psize{1.5cm}
\def\msize{1.35cm}
\def\w{.4cm} 
\begin{scope}[scale=0.5,xshift=-\msize,yshift=\msize]
\draw[fill=gray!30] (-\msize,\msize-\w)--(-\msize,\msize)--(-\msize+\w,\msize)--(\msize,-\msize+\w)--(\msize,-\msize)--(\msize-\w,-\msize)--cycle;
\draw[xshift=\msize,yshift=\msize,fill=gray!30] (0,0)--(0,-\w)--(-\w,0)--cycle;
\draw[xshift=-\msize,yshift=-\msize,fill=gray!30] (0,0)--(0,\w)--(\w,0)--cycle;
\end{scope}
\def\w{.8cm} 
\begin{scope}[scale=0.5,xshift=\msize,yshift=-\msize]
\foreach \b in {0,1}{
	\draw[fill=gray!30,xshift=-\msize+\b*\w,yshift=\msize-\b*\w] (0,0)--(\w,0)--(\w,-\w)--(0,-\w)--cycle;
}
\node at (.2,.2) [below right] {$\ddots$};
\end{scope}
\node at (-\psize,0) {$\left(\vphantom{\rule{\psize}{\psize}}\right.$};
\node at (\psize,0) {$\left.\vphantom{\rule{\psize}{\psize}}\right)$};
\end{tikzpicture}
\caption{Band matrix, block matrix, and their direct sum\label{fig:mat}}
\end{figure}

Before we turn to the delocalization phenomenon that will be studied in 
this paper, let us emphasize that delocalization of eigenvectors in the 
the classical sense cannot be universal with respect to the sparsity 
pattern.

\begin{ex}
\label{ex:exactnono}
Consider the three examples illustrated in Figure \ref{fig:mat}: an 
$N\times N$ band matrix or block matrix with $d$ nonzero entries 
per row, or their direct sum.

The situation for band matrices is subject to deep conjectures arising 
from mathematical physics. In particular, it is believed (cf.\ 
\cite{sodin,bourgade-survey} and the references therein) that this model 
should exhibit a localization-delocalization phase transition at $d\sim 
N^{5/6}$ for eigenvectors near the edge of the spectrum: these 
eigenvectors are expected to be delocalized when $d\gg N^{5/6}$, and to be 
localized when $d\ll N^{5/6}$.\footnote{%
	We use the notation $A\ll B$ to mean $A/B\to0$, and 
	$A\gg B$ to mean $A/B\to\infty$, as $N\to\infty$.}
These conjectures remain largely open to date.

On the other hand, it is clear that \emph{every} eigenvector of a random 
block matrix must be supported in one of the blocks. Thus the block matrix 
model has localized eigenvectors whenever $d=o(N)$. In fact, \cite[Theorem 
2.9]{bandsparse} shows that the eigenvectors in this model are nearly 
maximally localized (the latter result states that for any model of the 
kind considered in this paper, the mass of eigenvectors near the edge must 
be spread over at least $\sim \frac{d}{\log N}$ coordinates).

Finally, the set eigenvectors of the direct sum of a block matrix and a 
band matrix of width $d$ is the union of the eigenvectors of a block and 
a band matrix. In this example, it may be the case that half the 
eigenvectors are localized and half are delocalized. While contrived, such 
very simple examples already show that when we admit \emph{arbitrary} 
sparsity patterns, it is not even clear whether
classical eigenvector delocalization questions may 
be meaningfully formulated.
\end{ex}

As general $d$-regular graphs can exhibit arbitrarily complicated 
nonhomogeneities, the above examples suggest that the study of classical 
delocalization questions at this level of generlity is likely to be a 
hopeless task. Instead, the aim of this paper is to exhibit a new 
type of (de)localization phenomenon that is universal with respect to 
sparsity pattern. This phenomenon is necessary of a fundamentally 
different nature than classical eigenvector delocalization questions. In 
order to explain the nature of our main results, let us first recall that 
the macroscopic behavior of the \emph{eigenvalues} is in fact universal 
with respect to sparsity pattern.

\begin{thm}[Eigenvalues]
\label{thm:norm}
The following hold whenever $d,N\to\infty$.
\begin{enumerate}[a.]
\item The empirical eigenvalue distribution $\mu_N$
of $\frac{X_N}{\sqrt{d}}$ converges weakly in probability to the 
standard semicircle distribution:
$$
	\mu_N \Longrightarrow 
	\rho_\mathrm{sc}(x)\,dx := 
	\frac{1}{2\pi} \sqrt{4-x^2} 1_{|x|\le 2}\,dx
	\quad\mbox{in probability.}
$$
\item For every $\varepsilon>0$, we have 
$$
	\max\left\{(2-o(1))\sqrt{{d}},C\sqrt{\log N}\right\}\le 
	\mathbf{E}\|X_N\|\le 
	(2+\varepsilon)\sqrt{d}+K_\varepsilon\sqrt{\log N},
$$
where $C$ is a universal constant and $K_\varepsilon$ depends on 
$\varepsilon$ only.
\end{enumerate}
\end{thm}

\begin{proof}
Part a.\ and the upper bound of part b.\ are given in \cite[Theorem 
2.3]{bandsparse} and \cite[Theorem 1.1]{BvH}, respectively. The 
$(2-o(1))\sqrt{d}$ lower bound of part b.\ follows directly from part a., 
while the 
$C\sqrt{\log d}$ lower bound is given in \cite[Corollary 3.15]{BvH}.
\end{proof}

Theorem \ref{thm:norm} shows that the extreme eigenvalues of $X_N$ exhibit 
the following phase transition. When $d\gg\log N$, we have
$$
	\mathbf{E}\|X_N\| = (2+o(1))\sqrt{d},
$$
that is, the extreme eigenvalues of $X_N$ converge to the 
edge of the bulk (semicircle) eigenvalue distribution.
This behavior is analogous to that of Wigner matrices. In contrast, when 
$d\ll\log N$, we have
$$
	c\sqrt{\log N}\le\mathbf{E}\|X_N\| \le C\sqrt{\log N},
$$
that is, the extreme eigenvalues are of the same order as 
the size of the largest entry of $X_N$ (as the maximum of $n$ 
independent standard Gaussian variables is of order $\sqrt{\log n}$) and 
are separated from the bulk.

The above interpretation of the eigenvalue phase transition hints at an 
associated (de)localization phenomenon. When $d\gg\log N$, the behavior of 
the extreme eigenvalues is ``Wigner-like'', and one may expect to inherit 
some delocalization properties of Wigner matrices. In contrast, when 
$d\ll\log N$ the magnitude of the extreme eigenvalues is explained by the 
presence of exceptionally large matrix entries, which suggests that the 
extreme eigenvectors should be localized on the coordinates associated to 
these entries. Our main result will show that this phenomenon is captured 
not by (de)localization of the \emph{exact} top eigenvectors, which is 
ruled out by Example \ref{ex:exactnono}, but by that of \emph{approximate} 
top eigenvectors.

\begin{rmk}
Physics heuristics suggest that the transition between localization and 
delocalization of the exact eigenvectors coincides with a transition 
between Poissonian and random matrix statistics of the associated 
eigenvalues \cite{sodin,bourgade-survey}. In contrast, the 
localization-delocalization transition of this paper coincides with the 
transition where outlier eigenvalues detach from the bulk of the spectrum. 
As these outliers appear at the macroscopic scale, this phenomenon is much 
simpler and more robust than the fluctuations of the eigenvalues at the 
local scale.
\end{rmk}

\subsection{Main result}

Localization and delocalization of vectors can be described by various 
non-equivalent notions, such as the $\ell^\infty$ norm or other $\ell^p$ 
norms \cite{ESY,local1}, joint distribution of coordinates \cite{TaoVu1}, 
and no-gaps delocalization \cite{nogaps}; the survey 
\cite{eigenvectors-survey} includes results on several different notions 
of delocalization. Here we will use the notion of 
$(L,\kappa)$-delocalization used in \cite[\S7]{ESY}. A vector delocalized 
in this sense is one that has no ``peaks'' of mass $>\kappa^2$ in any 
subset of $\le L$ coordinates.

\begin{defn}[Delocalization]\label{def:deloc}
A real vector $v\in\Sb^{N-1}$ is \emph{$(L,\kappa)$-delocalized} if for 
every set $A\subseteq[N]$ of size $|A|\le L$, we have 
$\sum_{j\in A}v_j^2\le\kappa^2$. 
The set of $(L,\kappa)$-delocalized vectors will be denoted by
$$
	D_{L,\kappa}:=\bigg\{v\in\Sb^{N-1}:
	\sum_{j\in A}|v_j|^2\le\kappa^2
	\text{ for all }
	A\subseteq[N],|A|\le L
	\bigg\}.
$$
\end{defn}

The $(L,\kappa)$-delocalization condition becomes stricter for smaller 
$\kappa$ and larger $L$. We are primarily interested in the situation 
where $L$ is proportional to $N$. We will colloquially refer to a vector, 
or more precisely a sequence of vectors of increasing dimension 
$N\to\infty$, as \emph{delocalized} if it is $(\nu N,\kappa)$-delocalized 
for some $0<\nu,\kappa<1$ independent of $N$, and otherwise we will refer 
to it as \emph{localized}. In other words, a delocalized vector is one 
that is not concentrated in a vanishing fraction of the coordinates, 
while a localized vector is one that has a constant fraction of the mass 
concentrated in just a vanishing fraction of the coordinates.

Next, recall that any \emph{exact} top eigenvector $v$ (i.e., an 
eigenvector whose eigenvalue has the largest magnitude) of a self-adjoint 
matrix $X$ satisfies 
$$
	\|Xv\|_2 = \sup_{w\in\Sb^{N-1}}\|Xw\|_2 = \|X\|.
$$
Thus we define an \emph{approximate} top eigenvector as follows.

\begin{defn}[Approximate top eigenvector]
$v\in\Sb^{N-1}$ is an \emph{$(1-\varepsilon)$-approximate 
top eigenvector} of a self-adjoint matrix $X$ if
$\|Xv\|_2 \ge (1-\varepsilon)\|X\|$.
\end{defn}

We can now formulate our main result.

\begin{thm}[Localization-delocalization transition]
\label{thm:transition}
Let $N\to\infty$, and fix parameters $\varepsilon,\kappa,\nu$ 
that are independent of $N$. Then the following hold:
\medskip
\begin{enumerate}[(i)]
\item \emph{Localized regime $d\ll\log N$:} for
any $0<\varepsilon<1$ and $0<\kappa<1-\varepsilon$
$$
	\mathbf{P}[
	\text{every }(1-\varepsilon)\text{-approximate top eigenvector 
	of }X_N\text{ 
	is in }\mathbb{S}^{N-1}\backslash D_{o(1)N,\kappa}] =
	1-o(1).
$$
\item \emph{Delocalized regime $d\gg\log N$:} for any $0<\kappa<1$ and 
$0<\nu<\frac{c\kappa^2}{\log\frac{e}{\kappa}}$
$$
	\mathbf{P}[
	\text{there exists a }(1-o(1))\text{-approximate top eigenvector 
	of }X_N\text{ 
	in }D_{\nu N,\kappa}] =
	1-o(1).
$$
\end{enumerate}
\end{thm}

Qualitatively, Theorem \ref{thm:transition} yields a strong dichotomy 
between the localized and delocalized regimes. In the delocalized regime, 
there exists a $(1-o(1))$-approximate, that is, a ``nearly exact,'' top 
eigenvector that is delocalized. On the other hand, in the localized 
regime, every approximate top eigenvector that is even within a constant 
fraction of the edge of the spectrum must be localized.

A more careful interpretation of the quantitative aspect of Theorem 
\ref{thm:transition} further strenghtens this dichotomy. Recall that the 
benchmark example of a delocalized vector is a uniform random vector in 
$\mathbb{S}^{N-1}$; such random vectors arise as eigenvectors of GOE 
matrices, and therefore possess the strongest form of delocalization one 
could hope for. In Appendix~\ref{sec:unitsphere}, we will show that a 
uniform random vector in $\mathbb{S}^{N-1}$ is $(\nu 
N,\kappa)$-delocalized with high probability if and only if $\nu \lesssim 
\frac{\kappa^2}{\log\frac{e}{\kappa}}$. Thus the delocalized regime of 
Theorem \ref{thm:transition} yields an $(1-o(1))$-approximate top 
eigenvector that exhibits nearly the same degree of delocalization as a 
uniform random vector in the sphere, up to the value of the universal 
constant $c$.

It should be emphasized that Theorem \ref{thm:transition} sheds little 
light on classical questions in random matrix theory surrounding 
\emph{exact} eigenvectors: as is illustrated by Example 
\ref{ex:exactnono}, such questions are not even meaningful at the level of 
generality of this paper. While part (i) of Theorem \ref{thm:transition} 
yields localization of exact as well as approximate eigenvectors, it 
captures only a small subset of the regime in which some classical random 
matrix models (such as band matrices) exhibit eigenvector localization. On 
the other hand, part (ii) does not provide any information at all on 
individual eigenvectors, while Example \ref{ex:exactnono} illustrates that 
delocalization of approximate eigenvectors can arise even when all the 
exact eigenvectors are localized. The aim of Theorem \ref{thm:transition} 
is not to address such classical random matrix questions, but rather to 
exhibit a fundamentally different delocalization phenomenon that captures 
a nontrivial feature of a much larger class of random matrix models.

\subsection{Outline} 

This paper is organized as follows. In section~\ref{sec:loc}, 
we prove the localized regime of Theorem~\ref{thm:transition}. The basis 
of the proof is to show, using subgaussian estimates, that the existence 
of a delocalized approximate top eigenvector implies 
$\|X_N\|=O(\sqrt{d})$; consequently, no delocalized approximate top 
eigenvector can exist in the regime where $\|X_N\|\sim \sqrt{\log 
N}\gg\sqrt{d}$.

The remainder of the paper is devoted to the delocalized regime. The basic 
idea of the proof is to construct an approximate top eigenvector by taking 
a random superposition of many exact eigenvectors near the edge of the 
spectrum. In order to show such a vector is delocalized, we must establish 
that even though each exact eigenvector may be localized, the subspace 
spanned by sufficiently many of these eigenvectors is delocalized in an 
appropriate sense (in particular, this rules out that the exact 
eigenvectors are all simultaneously localized in a small subset of the 
coordinates). In section~\ref{sec:subspace}, we first explain what 
property of a linear subspace is needed to ensure that it contains 
delocalized vectors. We then show in section~\ref{sec:deloc} that this 
property is indeed satisfied for an appropriate eigenspace of $X_N$; this 
is accomplished by approximating the projection matrix of the eigenspace 
in terms of the resolvent, whose behavior is goverened by a mesoscopic 
form of the semicircle law. Combining all the above ingredients concludes 
the proof of Theorem \ref{thm:transition}.

Finally, Appendix~\ref{sec:unitsphere} shows that the delocalization 
provided by Theorem~\ref{thm:transition}(ii) agrees quantitatively with 
that of a uniform random vector in $\mathbb{S}^{N-1}$. This result is 
included to clarify the meaning of our main result, and is not used 
elsewhere.

\section{Proof of localization for \texorpdfstring{$d\ll\log N$}{d<<log N}}
\label{sec:loc}

The proof of localization Theorem~\ref{thm:transition}(i) is based on 
Gaussian concentration. Let us recall the general principle
for future reference, cf.\ \cite[\S5.4]{conc-ineq}.

\begin{thm}[Gaussian concentration]\label{thm:gaussian-concentration}
Let $Z$ be a standard Gaussian vector in $\mathbb{R}^n$, and let
$f:\R^n\to\R$ be $L$-Lipschitz with respect to the Euclidean norm. Then
$$
	\mathbf{P}\left[\left|f(Z)-\mathbb{E}[f(Z)]\right|\ge t\right]\le 
	2e^{-t^2/2L^2}
	\quad\text{ for all }t\ge0.
$$
\end{thm}

We need the following corollary.

\begin{cor}
\label{cor:eucconc}
For every $w\in\mathbb{R}^N$ such that $\|w\|_\infty\le 1$, we have
$$
	\mathbf{P}[
	\|X_N w\|_2 \ge \sqrt{dN}+t] 
	\le 
	2e^{-t^2/4d}.
$$
\end{cor}

\begin{proof}
By Cauchy-Schwarz, $|\langle z,w\rangle|\le\sqrt{d}\|z\|_2$ for 
any $z$ with $d$ nonzero entries. 
Thus any self-adjoint matrices $X,Y$ with $d$ nonzero 
entries in each row satisfy
$$
	|\|Xw\|_2-\|Yw\|_2| \le
	\|(X-Y)w\|_2 \le
	\sqrt{2d} \,
	\Bigg(\sum_{i\ge j} (X_{ij}-Y_{ij})^2\Bigg)^{1/2}.
$$
We can therefore view $\|X_Nw\|_2$ as a $\sqrt{2d}$-Lipschitz function
of the i.i.d.\ standard Gaussian variables $(g_{xy})_{x\ge y,x\sim y}$
that appear in its definition. 
The conclusion follows from
Theorem \ref{thm:gaussian-concentration} and $\mathbf{E}\|X_N w\|_2\le 
(\mathbf{E}\|X_Nw\|_2^2)^{1/2}=\sqrt{d}\|w\|_2\le\sqrt{dN}$.
\end{proof}

The key idea behind the localized regime is the following.

\begin{lem}
\label{lem:locest}
Let $0<\varepsilon<1$, $0<\kappa<1-\varepsilon$, $\frac{1}{N}\le \nu<1$.
Define the event
$$
	\Omega_{\nu,\kappa,\varepsilon,N} :=
	\{\text{there exists a }(1-\varepsilon)\text{-approximate
	top eigenvector of }X_N\text{ in }D_{\nu N,\kappa}\}.
$$
Then
$$
	\mathbf{P}[\Omega_{\nu,\kappa,\varepsilon,N}\cap\{\|X_N\| > 
	\tfrac{8}{1-\varepsilon-\kappa}(\tfrac{d}{\nu})^{1/2}
	\}] \le
	2 e^{-3N}.
$$
\end{lem}

\begin{proof}
Let $L=\lfloor \nu N\rfloor$, and define for any $v\in\mathbb{S}^{N-1}$
the vectors $v^+,v^-\in\mathbb{R}^N$ as
$$
	v^+_i := v_i\mathbbm{1}_{|v_i|>L^{-1/2}},\qquad
	v^-_i := v_i\mathbbm{1}_{|v_i|\le L^{-1/2}}.
$$
As $|\{i:|v_i|>L^{-1/2}\}|\le L$
for every $v\in\mathbb{S}^{N-1}$, we can estimate
$$
	\|v^+\|_2^2 =
	\sum_{i:|v_i|>L^{-1/2}}|v_i|^2
	\le
	\kappa^2,\qquad
	\|v^-\|_\infty \le L^{-1/2}
$$
for all $v\in D_{L,\kappa}=D_{\nu N,\kappa}$. We readily obtain
\begin{align*}
	\sup_{v\in D_{L,\kappa}}\|X_Nv\|_2 &\le
	\sup_{v\in D_{L,\kappa}}\|X_Nv^-\|_2 +
	\sup_{v\in D_{L,\kappa}}\|X_Nv^+\|_2 \\ &\le
	L^{-1/2}\sup_{\|w\|_{\infty}\le 1}\|X_Nw\|_2 +
	\kappa\|X_N\|.
\end{align*}
In particular, as
$\sup_{v\in D_{L,\kappa}}\|X_Nv\|_2 \ge (1-\varepsilon)\|X_N\|$
on $\Omega_{\nu,\kappa,\varepsilon,N}$, we can estimate
$$
	\mathbf{P}[\Omega_{\nu,\kappa,\varepsilon,N}\cap\{\|X_N\| > t\}] \le
	\mathbf{P}\bigg[
	\frac{L^{-1/2}}{1-\varepsilon-\kappa}
	\sup_{\|w\|_{\infty}\le 1}\|X_Nw\|_2 > t\bigg].
$$
Now note that as $w\mapsto \|X_Nw\|_2$ is convex, the supremum 
over $w$ is attained at one of the extreme points $\{-1,+1\}^N$ of 
the unit cube. Thus a union bound yields
$$
	\mathbf{P}[\Omega_{\nu,\kappa,\varepsilon,N}\cap\{\|X_N\| > t\}] \le
	\sum_{w\in \{-1,+1\}^N}
	\mathbf{P}[
	\|X_Nw\|_2 > (1-\varepsilon-\kappa)L^{1/2}t],
$$
and the conclusion follows from Corollary \ref{cor:eucconc}
and $\lfloor \nu N\rfloor\ge \frac{\nu N}{2}$ for $\nu \ge \frac{1}{N}$.
\end{proof}

We can now complete the proof of the localized regime.

\begin{proof}[Proof of Theorem~\ref{thm:transition}(i)]
Assume that $d\ll\log N$, and fix $\beta>0$ independent of $N$.
Let $\nu=\nu(N):=\frac{d}{\beta^2\log N}=o(1)$, and note that
$\frac{1}{N}\le\nu<1$ for $N$ sufficiently large as $d\ge 1$.
We can therefore apply Lemma \ref{lem:locest} to estimate
\begin{multline*}
	\mathbf{P}[
	\text{there exists a }(1-\varepsilon)\text{-approximate
	top eigenvector of }X_N\text{ in }D_{\nu(N) N,\kappa}]
	\\ \le 2e^{-3N} + \mathbf{P}\big[\|X_N\| \le
	\tfrac{8}{1-\varepsilon-\kappa}\beta\sqrt{\log N}\big].
\end{multline*}
The proof is therefore complete once we show that the probability on the 
right-hand side is $o(1)$ for a sufficiently small choice of $\beta$.
To this end, note that as $\|X-Y\|\le 
\sqrt{2}\,(\sum_{i\ge j}(X_{ij}-Y_{ij})^2)^{1/2}$ for any self-adjoint 
matrices $X,Y$, the random variable $\|X_N\|$ may be viewed as a 
$\sqrt{2}$-Lipschitz function of 
the underlying i.i.d.\ standard Gaussian variables $(g_{xy})_{x\ge y}$.
Thus Theorem \ref{thm:gaussian-concentration} yields
\begin{equation}
\label{eq:normconc}
	\mathbf{P}[|\|X_N\|- \mathbf{E}\|X_N\|| \ge t] \le 2e^{-t^2/4}
	\quad\text{for all }t>0.
\end{equation}
As $\mathbf{E}\|X_N\|\ge C\sqrt{\log N}$ by Theorem \ref{thm:norm}, it 
follows that
$$
	\mathbf{P}\big[\|X_N\| \le \tfrac{1}{2}C\sqrt{\log N}\big] \le 
	2 N^{-C^2/16}=o(1),
$$
concluding the proof.
\end{proof}

\section{Delocalization of uniform random vectors in a subspace}
\label{sec:subspace}

We now turn to the delocalized regime of Theorem \ref{thm:transition}(ii), 
whose proof will occupy the remainder of this paper. In the regime $d\gg 
\log N$, Theorem \ref{thm:norm} ensures that the largest eigenvalue of 
$X_N$ sticks to the bulk of the spectrum. Consequently,
any superposition of the top $o(N)$ exact eigenvectors of $X_N$ will yield 
a $(1-o(1))$-approximate top eigenvector. The basic idea behind the proof 
is to show that even though each exact eigenvector may itself be 
localized, we can always find a superposition of $o(N)$ exact eigenvectors 
that is delocalized.

In order for such a strategy to succeed, the exact eigenvectors must 
exhibit at least the following two qualitative features:
\begin{enumerate}[1.]
\itemsep\abovedisplayskip
\item Each exact eigenvector must be at least somewhat delocalized: if 
each exact eigenvector were concentrated on $O(1)$ coordinates, then any 
superposition of $o(N)$ such eigenvectors would always be localized.
\item While individual exact eigenvectors may be localized, the top $o(N)$
exact eigenvectors cannot be simultaneously localized in the same subset 
of coordinates. In other words, the locations where different exact 
eigenvectors are localized must be spread out across all the coordinates.
\end{enumerate}
Property 1.\ was previously established in \cite[Theorem 2.9]{bandsparse},
which states that the exact eigenvectors near the edge of the spectrum 
must be spread over at least $\sim \frac{d}{\log N}$ coordinates (this 
result will not be used in our proofs). However, this does not suffice to
ensure Property 2., which requires us to understand delocalization 
of the entire space spanned by the top $o(N)$ eigenvectors.

In this section, we begin our analysis by explaining what property of a 
linear subspace $E\subseteq\mathbb{R}^N$ is needed to ensure that $E$ 
contains a delocalized unit vector. The main result of this section is the 
following.

\begin{prop}[Delocalized subspace]\label{prop:delocspace}
Let $E\subseteq\mathbb{R}^N$ be a linear subspace of dimension $m$, and 
denote by $P_E$ the orthogonal projection onto $E$. Suppose that
$$
	\max_{x\in[N]} (P_E)_{xx} \le
	\frac{Cm}{N}
$$
for some constant $C>0$. Then there exists a constant $c>0$ that depends 
only on $C$ such that $E\cap D_{\nu N,\kappa}\ne\varnothing$ for every
$0<\kappa<1$ and $0<\nu<\frac{c\kappa^2}{\log\frac{e}{\kappa}}$.
\end{prop}

Proposition \ref{prop:delocspace} shows that the relevant delocalization 
property of a linear subspace is control of the diagonal entries of its 
projection matrix. The remainder of the proof of Theorem 
\ref{thm:transition} will then aim to show that this property holds when 
$E$ is taken to be the linear span of the top $o(N)$ exact eigenvectors of 
$X_N$.

Let us first turn to the proof of Proposition \ref{prop:delocspace}. 
Rather than establish delocalization in the sense of Definition 
\ref{def:deloc} directly, it will be more convenient to establish 
$\ell^q$-bounds. A simple lemma shows that the former is implied by
the latter.

\begin{lem}
\label{lem:lq}
For all $v\in\mathbb{S}^{N-1}$, $L\le N$, and $q\ge 2$, we have
$v\in D_{L,L^{1/2-1/q}\|v\|_q}$.
\end{lem}

\begin{proof}
It suffices to note that for any $A\subseteq[N]$ with $|A|\le L$, we 
have
$$
	\sum_{j\in A}|v_j|^2 \le
	L^{1-2/q}
	\Bigg(\sum_{j\in A}|v_j|^q\bigg)^{2/q}
	\le L^{1-2/q} \|v\|_q^2
$$
by H\"older's inequality.
\end{proof}

To prove Proposition \ref{prop:delocspace}, we will bound the 
$\ell^q$-norm of a uniformly chosen random vector in 
$E\cap\mathbb{S}^{N-1}$; the conclusion then follows from Lemma 
\ref{lem:lq} by optimizing over $q$. Proposition \ref{prop:delocspace} may 
therefore be viewed as a complement to the delocalization established in 
Appendix~\ref{sec:unitsphere} for a uniformly chosen random vector in the 
entire sphere $\mathbb{S}^{N-1}$: here we show that a uniformly chosen 
unit vector in a subspace $E$ is still delocalized when $E$ satisfies the 
requisite assumption. (We do not develop high probablity results as in 
Appendix~\ref{sec:unitsphere}, as these are not needed in the sequel.)

\begin{proof}[Proof of Proposition \ref{prop:delocspace}]
To estimate the $\ell^q$-norm of a uniformly chosen random vector in 
$E\cap\Sb^{N-1}$, let $Z\sim N(0,P_E)$ be a Gaussian random vector in $\mathbb{R}^N$ with 
zero mean and covariance matrix $P_E$, i.e., a standard Gaussian vector in 
$E$. Then $V:=\frac{Z}{\|Z\|_2}$ is 
uniformly distributed on $E\cap\mathbb{S}^{N-1}$ and $\|Z\|_2$ and $V$ 
are independent (as the law of $Z$ is rotationally invariant in $E$).
Therefore
$$
	\mathbf{E}\|Z\|_2^q \cdot \mathbf{E}\|V\|_q^q =
	\mathbf{E}\|Z\|_q^q = 
	\sum_{x=1}^N \mathbf{E} |Z_x|^q.
$$
Now note that $Z_x\sim N(0,(P_E)_{xx})$ for every $x$, so
that $\mathbf{E}|Z_x|^q \le 2 q^{q/2} (P_E)_{xx}^{q/2}$ for all $q\ge 2$
\cite[Theorem 2.1]{conc-ineq}. On the other hand, we have by Jensen's 
inequality 
$$
	(\mathbf{E}\|Z\|_2^q)^{2/q} \ge \mathbf{E}\|Z\|_2^2
	= \mathrm{Tr}(P_E) = m.
$$
Combining the above estimates with the assumption on $P_E$ yields
$$
	\mathbf{E}\|V\|_q^q
	\le
	2 (Cq)^{q/2} N^{1-q/2} \quad\text{for all }q\ge 2.
$$
In particular, there exists $v\in 
E\cap\mathbb{S}^{N-1}$ so that $\|v\|_q^q\le 2 (Cq)^{q/2} N^{1-q/2}$.

Now fix $0<\nu<1$ and let $q=2\log\frac{e}{\nu}$. Then there exists
$v\in E\cap\mathbb{S}^{N-1}$ so that
$$
	(\nu N)^{1/2-1/q}
	\|v\|_q \le
	2^{1/q} (Cq)^{1/2} \nu^{1/2-1/q}
	\le
	(4eC)^{1/2}
	\sqrt{\nu \log\frac{e}{\nu}}.
$$
Applying Lemma \ref{lem:lq} shows that whenever $0<\nu<1$ satisfies
$4eC \nu \log\frac{e}{\nu}\le \kappa^2$, there exists a vector
$v\in E\cap D_{\nu N,\kappa}$. The conclusion follows readily.
\end{proof}

\section{Proof of delocalization for \texorpdfstring{$d\gg\log N$}{d>>log N}}
\label{sec:deloc}

To complete the proof of Theorem \ref{thm:transition}(ii), it remains to 
show that the condition of Proposition \ref{prop:delocspace} holds for the 
space $E$ spanned by the top $o(N)$ eigenvectors of $X_N$. To this end, we 
first approximate the projection matrix $P_E$ in terms of the resolvent of 
$X_N$. We can then apply resolvent estimates for nonhomogeneous random 
matrices to deduce the requisite delocalization property.

Let us begin by formalizing the projection matrix approximation.

\begin{lem}[Projection matrix approximation]
\label{lem:projapprox}
Let $X$ be a self-adjoint matrix, and let $E_{[a,b]}$ be the space spanned 
by the eigenvectors of $X$ with eigenvalues in $[a,b]$. Then for any
$a<b$ and $\gamma\ge\delta>0$, we can estimate
$$
	P_{E_{[a,b]}} \le
	\bigg(1+\frac{2\delta}{\gamma}\bigg)
	\frac{1}{\pi}\Im \int_{a-\gamma}^{b+\gamma}
	(X-\lambda-i\delta)^{-1}\,d\lambda
$$
and
$$
	\frac{1}{\pi}\Im \int_{a+\gamma}^{b-\gamma}
	(X-\lambda-i\delta)^{-1}\,d\lambda
	\le
	P_{E_{[a,b]}}
	+
	\frac{\delta}{\pi\gamma}\mathbbm{1}
$$
in the positive semidefinite order \emph{(}here $\Im X := 
\frac{X-X^*}{2i}$\emph{)}.
\end{lem}

\begin{proof}
We begin by estimating
\begin{multline*}
	\int_{a-\gamma}^{b+\gamma}
	\Im 
	\frac{1}{x-\lambda-i\delta}\,d\lambda = 
	\int_{a-\gamma}^{b+\gamma}
	\frac{\delta}{(x-\lambda)^2+\delta^2}\,d\lambda
	\\ =
	\tan^{-1}\bigg(\frac{b+\gamma-x}{\delta}\bigg) -
	\tan^{-1}\bigg(\frac{a-\gamma-x}{\delta}\bigg) 
	\ge
	2\tan^{-1}\bigg(\frac{\gamma}{\delta}\bigg)
        \,1_{[a,b]}(x)
\end{multline*}
for all $x\in\mathbb{R}$,
where we used that $b+\gamma-x \ge \gamma$ and $x+\gamma-a\ge \gamma$ for 
$a\le x\le b$.
Now note that as
$g(s):=\tan^{-1}(\frac{1}{s})$ is convex for $s\ge 0$ and satisfies
$g(0^+)=\frac{\pi}{2}$, $g'(0^+)=-1$, we can estimate
$\tan^{-1}(t)\ge \frac{\pi}{2}-\frac{1}{t}$ for all $t\ge 0$. Thus
$$
	1_{[a,b]}(x) \le
	\bigg(1+\frac{2\delta}{\gamma}\bigg)
	\frac{1}{\pi}\Im
	\int_{a-\gamma}^{b+\gamma}
	\frac{1}{x-\lambda-i\delta}\,d\lambda 
$$
for all $x\in\mathbb{R}$,
where we used $1-\frac{2}{\pi}u \ge \frac{1}{1+2u}$ for $|u|\le 1$. 

In the opposite direction, an analogous computation yields
\begin{multline*}
	\int_{a+\gamma}^{b-\gamma}
	\Im 
	\frac{1}{x-\lambda-i\delta}\,d\lambda
	=
	\tan^{-1}\bigg(\frac{b-\gamma-x}{\delta}\bigg) -
	\tan^{-1}\bigg(\frac{a+\gamma-x}{\delta}\bigg)
	\\
	\le 
	\pi 1_{[a,b]}(x) + 
	\bigg(
	\frac{\pi}{2}-
	\tan^{-1}\bigg(\frac{\gamma}{\delta}\bigg)\bigg)1_{[a,b]^c}(x)
	\le
	\pi 1_{[a,b]}(x) + \frac{\delta}{\gamma}
\end{multline*}
for all $x\in\mathbb{R}$,
where we used 
$\tan^{-1}(t)\le\frac{\pi}{2}$ and
$\tan^{-1}(t)\ge \frac{\pi}{2}-\frac{1}{t}$, respectively.

Applying these inequalities to $X$ yields the conclusion 
by functional calculus.
\end{proof}

The reason Lemma \ref{lem:projapprox} is useful for our purposes is that 
we can compute the resolvent of $X_N$ by means of a mesoscopic semicircle 
law. The proof of the following result is an elementary application of the 
``intrinsic freeness'' theory of \cite{free}. (In the special case of 
interest here such a bound could alternatively be obtained, albeit with 
considerably more effort, by adapting the methods of  \cite[\S3]{erdosrenyi}.)

\begin{lem}[Mesoscopic semicircle law]
\label{lem:free}
Denote by
$$
	m_{\rm sc}(z) := -\frac{z}{2}+\frac{\sqrt{z^2-4}}{2}
$$
the Stieltjes transform of the standard semicircle law. Then
$$
	\| \mathbf{E} [(d^{-1/2}X_N-z)^{-1}] -
	m_{\rm sc}(z)\mathbbm{1}_N\| \le
	\frac{2}{d\,(\Im z)^5}
$$
for every $z\in\mathbb{C}$ with $\Im z>0$.
\end{lem}

\begin{proof}
The theory of \cite{free} enables us to compare the spectral statistics of 
a random matrix $X_N$ with those of a certain deterministic operator
$X_{N,\textrm{free}}$ that arises from free probability theory.\footnote{%
More precisely, $X_{N,\mathrm{free}}$ may be defined as a matrix whose 
entries are $(X_{N,\mathrm{free}})_{xy}= 1_{x\sim y}s_{xy}$, where
$s_{xy}$ are freely independent semicircular variables modulo symmetry $s_{yx}=s_{xy}$.} 
In particular, \cite[Theorem 2.8 and Lemma 3.1]{free} yield
$$
	\| \mathbf{E} [(d^{-1/2}X_N-z)^{-1}] -
	(\mathrm{id}\otimes\tau)[(d^{-1/2}X_{N,\textrm{free}}-z)^{-1}]\|
	\le
	\frac{2}{d}
	\frac{1}{(\Im z)^5}.
$$
It remains to compute
$G(z) := (\mathrm{id}\otimes\tau)[(d^{-1/2}X_{N,\textrm{free}}-z)^{-1}]$.

To this end, note that by \cite[eq.\ (1.5)]{free-dyson},
$G(z)$ satisfies the matrix Dyson equation
$$
	\frac{1}{d}
	\sum_{\{i,j\}:i\sim j}
	E_{ij} G(z) E_{ij} + G(z)^{-1} + z \mathbbm{1}_N  = 0,
$$
where $E_{ij}:=e_ie_j^*+1_{j\ne i}e_je_i^*$. Moreover, \cite[Theorem 
2.1]{hfs} states that for any $z\in\mathbb{C}$ with $\mathrm{Im}\,z>0$, 
the matrix Dyson equation has a unique solution with positive imaginary 
part. As $m_{\rm sc}(z)$ satisfies
the equation $m_{\rm sc}(z)+ m_{\rm sc}(z)^{-1} + z=0$ and
$\Im m_{\rm sc}(z)>0$ whenever $\Im z>0$, it is readily verified that
$G(z)=m_{\rm sc}(z)\mathbbm{1}_N$ is the unique solution of the equation 
in the present setting.
\end{proof}

At this point we can clearly see the origin of the delocalization 
phenomenon: Lemma \ref{lem:free} shows that the resolvent behaves to 
leading order as a multiple of the identity matrix; by Lemma 
\ref{lem:projapprox}, this will ensure that all diagonal entries of the 
projection matrix of a suitable eigenspace of $X_N$ are roughly of the 
same order, which is precisely what is needed to apply Proposition 
\ref{prop:delocspace}. 

Before we proceed to implementing this program, we must establish 
concentration of the diagonal entries of the resolvent. This will be 
needed in order to upgrade the expected resolvent bound of
Lemma \ref{lem:free} to a high probability bound.

\begin{lem}[Resolvent concentration]
\label{lem:rconc}
Define $G_N(z):=(d^{-1/2}X_N-z)^{-1}$.
Then for any $a<b$ and $t,\delta>0$, we have
$$
	\mathbf{P}\bigg[
	\bigg|
	\Im \int_a^b 
	G_N(\lambda+i\delta)_{xx}d\lambda -
	\Im \int_a^b 
	\mathbf{E} G_N(\lambda+i\delta)_{xx}d\lambda
	\bigg|
	\ge
	\frac{t}{\delta^2} \frac{b-a}{\sqrt{d}}
	\bigg]
	\le 2e^{-t^2/4}.
$$
\end{lem}

\begin{proof}
Using the resolvent identity, we have
$$
	\|(X-z)^{-1}-(Y-z)^{-1}\| = \|(X-z)^{-1}(Y-X)(Y-z)^{-1}\|
	\le \frac{\|Y-X\|}{(\Im z)^2}.
$$
Thus
$$
	f(X):= \Im \int_a^b 
	[(d^{-1/2}X-\lambda-i\delta)^{-1}]_{xx}\,d\lambda
$$
is
$(b-a)d^{-1/2}\delta^{-2}$-Lipschitz with respect to the operator norm.
In particular, as $\|Z\|\le (\sum_{ij} |Z_{ij}|^2)^{1/2} \le
\sqrt{2}\,(\sum_{i\ge j} |Z_{ij}|^2)^{1/2}$ for any self-adjoint matrix 
$Z$, we may view $f(X_N)$ as a
$(b-a)\sqrt{2}\,d^{-1/2}\delta^{-2}$-Lipschitz function of the i.i.d.\ 
standard Gausian variables $(g_{xy})_{x\ge y}$ that define $X_N$.
The conclusion follows by Theorem \ref{thm:gaussian-concentration}.
\end{proof}

We can now combine the above results.

\begin{cor}[Projection matrix estimate]
\label{cor:projest}
Let $0\le a<b\le 3$ and $0<\delta\le\gamma\le 1$. Denote by 
$P$ the projection matrix of the space spanned by the eigenvectors of 
$d^{-1/2}X_N$ with eigenvalues in $[a,b]$. Then we have
$$
	\mathbf{P}\bigg[
	\max_{x\in[N]}
	P_{xx}
	\ge
	\bigg(1+\frac{2\delta}{\gamma}\bigg)\frac{1}{\pi}
	\bigg\{
	\Im \int_{a-\gamma}^{b+\gamma} 
	m_{\rm sc}(\lambda+i\delta) 
	d\lambda
	+\frac{10}{\delta^5d} 
	+
	\frac{20\sqrt{\log N}}{\delta^2\sqrt{d}}
	\bigg\}
	\bigg]
	\le \frac{2}{N^3}
$$
and
$$
	\mathbf{P}\bigg[
	\min_{x\in[N]}
	P_{xx} 
	\le
	\frac{1}{\pi}
	\bigg\{
	\Im \int_{a+\gamma}^{b-\gamma} 
	m_{\rm sc}(\lambda+i\delta) 
	d\lambda
	- \frac{\delta}{\gamma}
	- \frac{6}{\delta^5 d}
	- \frac{12\sqrt{\log N}}{\delta^2\sqrt{d}}
	\bigg\}
	\bigg]
	\le \frac{2}{N^3}.
$$
\end{cor}

\begin{proof}
Note first that Lemmas \ref{lem:projapprox} and \ref{lem:free} yield
\begin{gather*}
	P_{xx} \le
	\bigg(1+\frac{2\delta}{\gamma}\bigg)
	\frac{1}{\pi} \Im \int_{a-\gamma}^{b+\gamma}
	G_N(\lambda+i\delta)_{xx}d\lambda,\\
	\Im \int_{a-\gamma}^{b+\gamma}
	\mathbf{E}G_N(\lambda+i\delta)_{xx}d\lambda
	\le
	\Im \int_{a-\gamma}^{b+\gamma}
	m_{\rm sc}(\lambda+i\delta)d\lambda +
	\frac{10}{\delta^5d},
\end{gather*}
where we used that $b-a+2\gamma \le 5$. Thus Lemma \ref{lem:rconc} yields
$$
	\mathbf{P}\bigg[
	P_{xx}
	\ge
	\bigg(1+\frac{2\delta}{\gamma}\bigg)\frac{1}{\pi}
	\bigg\{
	\Im \int_{a-\gamma}^{b+\gamma} 
	m_{\rm sc}(\lambda+i\delta) 
	d\lambda
	+\frac{10}{\delta^5d} 
	+
	\frac{5t}{\delta^2\sqrt{d}}
	\bigg\}
	\bigg]
	\le 2e^{-t^2/4}
$$
for all $x\in[N]$ and $t\ge 0$. The first inequality now follows by taking 
a union bound and choosing $t=4\sqrt{\log N}$. The second inequality is 
derived analogously.
\end{proof}

The final ingredient that will be needed in the proof of Theorem 
\ref{thm:transition}(ii) is an estimate on the integrals that appear in 
Corollary \ref{cor:projest}.

\begin{lem}[Semicircle law estimate]
\label{lem:mscint}
For any $0<\delta<\varepsilon<1$ and $c\ge 2+\varepsilon$
$$
	\bigg(1-\frac{2\delta}{\varepsilon}\bigg)
	\frac{\sqrt{3}}{3\pi}
	\varepsilon^{3/2}
	\le
	\frac{1}{\pi} \Im \int_{2-2\varepsilon}^c
	m_{\rm sc}(\lambda+i\delta)d\lambda
	\le 2\varepsilon^{3/2} + \frac{\delta}{\varepsilon}.
$$
\end{lem}

\begin{proof}
As $m_{\rm sc}(z)  = \int (x-z)^{-1}\rho_{\rm sc}(x)\,dx$ is the Stieltjes 
transform of the standard semicircle law (as defined in Theorem 
\ref{thm:norm}), we can apply precisely the same estimates as in
Lemma \ref{lem:projapprox} to estimate for any $a<b$
$$
	\frac{1}{\pi}
	\Im
	\int_{a+\varepsilon}^{b-\varepsilon}
	m_{\rm sc}(\lambda+i\delta)d\lambda
	-
	\frac{\delta}{\varepsilon}
	\le
	\int_a^b \rho_{\rm sc}(x)dx
	\le
	\bigg(1+\frac{2\delta}{\varepsilon}\bigg) \frac{1}{\pi}
	\Im
	\int_{a-\varepsilon}^{b+\varepsilon}
	m_{\rm sc}(\lambda+i\delta)d\lambda.
$$
Choosing 
$a=2-3\varepsilon$, $b=c+\varepsilon$ and
$a=2-\varepsilon$, $b=c-\varepsilon$, respectively, yields
$$
	\bigg(1-\frac{2\delta}{\varepsilon}\bigg) 
	\int_{2-\varepsilon}^2 \rho_{\rm sc}(x)dx
	\le
	\frac{1}{\pi}
	\Im
	\int_{2-2\varepsilon}^c
	m_{\rm sc}(\lambda+i\delta)d\lambda
	\le
	\int_{2-3\varepsilon}^2 \rho_{\rm sc}(x)dx
	+
	\frac{\delta}{\varepsilon},
$$
where we used that $\rho_{\rm sc}$ is supported on $[-2,2]$ and that
$1-u \le \frac{1}{1+u}$ for $u\ge 0$. It remains to note that we can 
estimate 
$$
	\frac{2\sqrt{1-u/4}}{3\pi}
	u^{3/2}
	\le
	\int_{2-u}^2 \rho_{\rm sc}(x)dx =
	\frac{1}{2\pi}
	\int_{2-u}^2 \sqrt{(2-x)(2+x)} dx 
	\le
	\frac{2}{3\pi}
	u^{3/2}
$$
for $u\le 4$, and the proof is readily completed.
\end{proof}

We are now ready to complete the proof of Theorem 
\ref{thm:transition}(ii).

\begin{proof}[Proof of Theorem \ref{thm:transition}(ii)]
We will assume throughout the proof that $d\gg\log N$. 
Let $E_N$ be the space spanned by the eigenvectors of $d^{-1/2}X_N$ with 
eigenvalues in the interval $[2-\varepsilon,3]$, and denote by $P_N$ its 
projection matrix. Here $\varepsilon=\varepsilon_N=o(1)$ will depend on 
$N$ in a manner that will be chosen below. As
$$
	\mathbf{P}\bigg[d^{-1/2}\|X_N\| \ge (2+s) + 
	K_s\sqrt{\frac{\log N}{d}} + t\bigg]
	\le 2e^{-dt^2/4}
$$
for any $s,t>0$ by \eqref{eq:normconc} and Theorem \ref{thm:norm},
we have 
$$
	\mathbf{P}[d^{-1/2}\|X_N\|\le 2+o(1)]=1-o(1).
$$
Thus every unit vector in $E_N$ is a $(1-o(1))$-approximate top 
eigenvector of $X_N$ with probability $1-o(1)$. It remains to show that 
$E_N$ contains delocalized vectors.

To this end, we begin by applying Corollary \ref{cor:projest} and
Lemma \ref{lem:mscint}
with $a=2-\varepsilon$, $b=3$, $\gamma=\frac{\varepsilon}{2}$, and
$\delta=\varepsilon^3$. This yields 
$$
	\mathbf{P}\bigg[
	\max_{x\in[N]}
	(P_N)_{xx}
	\ge
	C\varepsilon^{3/2} 
	+\frac{C}{\varepsilon^{15}d} 
	+\frac{C}{\varepsilon^6}
	\sqrt{\frac{\log N}{d}}
	\bigg]
	\le \frac{2}{N^3}
$$
and
$$
	\mathbf{P}\bigg[
	\min_{x\in[N]}
	(P_N)_{xx} 
	\le
	c\varepsilon^{3/2}
	- \frac{C}{\varepsilon^{15} d}
	- \frac{C}{\varepsilon^6} \sqrt{\frac{\log N}{d}}
	\bigg]
	\le \frac{2}{N^3}
$$
for sufficiently large $N$, where $c,C>0$ are universal constants.
In particular, if we choose $\varepsilon = (\frac{\log N}{d})^{1/17}$,
the last two terms inside each of the above probabilities become 
negligible, and we can conclude that
$$
	\mathbf{P}\bigg[
	\max_{x\in[N]}(P_N)_{xx} \le
	C'
	\min_{x\in[N]}(P_N)_{xx} 
	\bigg]=1-o(1)
$$
for a universal constant $C'$.

Now let $m_N=\dim E_N$ (note that this is a random variable). Then
$$
	\min_{x\in[N]}(P_N)_{xx}
	\le
	\frac{1}{N}\sum_{x\in[N]} (P_N)_{xx}
	=
	\frac{\mathrm{Tr}(P_N)}{N} =
	\frac{m_N}{N}.
$$
Thus we have shown that the assumption of Proposition 
\ref{prop:delocspace} holds with probability $1-o(1)$, completing the 
proof of the theorem.
\end{proof}

\appendix

\section{Delocalization of uniform random vectors in
\texorpdfstring{$\mathbb{S}^{N-1}$}{the sphere}}
\label{sec:unitsphere}

The aim of this section is to elucidate the delocalization properties of 
uniform random vectors in $\mathbb{S}^{N-1}$. The main result of this 
section is the following.

\begin{prop}[Unit sphere delocalization]
\label{prop:unitsphere}
Let $V_N$ be a random vector that is uniformly distributed on 
$\mathbb{S}^{N-1}$, and fix $0<\kappa<1$ that is independent of $N$. Then 
there exist universal constants $c_1,c_2>0$ such that the following hold 
as $N\to\infty$.
\begin{enumerate}[a.]
\item If $\nu \le \frac{c_1\kappa^2}{\log\frac{e}{\kappa}}$, then
$\mathbf{P}[V_N\in D_{\nu N,\kappa}]=1-o(1)$.
\item If $\nu \ge \frac{c_2\kappa^2}{\log\frac{e}{\kappa}}$, then
$\mathbf{P}[V_N\in D_{\nu N,\kappa}]=o(1)$.
\end{enumerate}
\end{prop}

In particular, this result shows that the approximate top eigenvector 
provided by Theorem \ref{thm:transition} in the delocalized regime is 
essentially as delocalized as a uniform random vector, up to to the value 
of the universal constant.

We begin by introducing some notation. For any 
vector $z\in\mathbb{R}^N$, we denote by $z_{(1)}\ge z_{(2)}\ge\cdots\ge 
z_{(N)}\ge 0$ the decreasing rearrangement of $z$, that is, the absolute 
values of the entries of $z$ sorted in decreasing order. We define the 
norm
$$
	\|z\|_{(L)} :=
	\sup_{A\subseteq[N]:|A|\le L}
	\Bigg(
	\sum_{j\in A}|z_j|^2\Bigg)^{1/2} =
	\Bigg(\sum_{k=1}^{\lfloor L\rfloor} z_{(k)}^2\Bigg)^{1/2}.
$$
Thus $z\in D_{L,\kappa}$ if and only if
$\|z\|_{(L)}\le\kappa$. The proof of Proposition \ref{prop:unitsphere} 
requires us to estimate $\|V_N\|_{(\nu N)}$. Let us first 
investigate its expectation.

\begin{lem}
\label{lem:sphorder}
There exist universal constants $C_1,C_2>0$ so that
$$
	C_1\nu\log\frac{e}{\nu}
	\le
	\mathbf{E}\|V_N\|_{(\nu N)}^2 \le
	C_2\nu\log\frac{e}{\nu}.
$$
\end{lem}

\begin{proof}
As $\frac{1}{3}\le \|z\|_{(N/2)}^2\le \|z\|_{(\nu N)}^2\le 
1$ for all $\nu\ge\frac{1}{2}$ and $z\in \mathbb{S}^{N-1}$, it suffices to 
consider $\nu<\frac{1}{2}$.
Let $Z$ be a standard Gaussian vector in $\mathbb{R}^N$. Then well-known 
estimates on order statistics (see, e.g., \cite[Theorem 2.5]{BT12}) yield
$$
	c\log\frac{eN}{k}\le\mathbf{E}Z_{(k)}^2\le 
	C\log\frac{eN}{k}
$$
for all $1\le k\le \frac{N}{2}$, where $c,C>0$ are universal constants.
Stirling's formula yields
$$
	c'\nu N\log\frac{e}{\nu}
	\le
	\mathbf{E}\|Z\|_{(\nu N)}^2
	\le C'\nu N\log
	\frac{e}{\nu}
$$
for universal constants $c',C'>0$. It remains to 
note that $Z\stackrel{\mathrm{d}}{=}\|Z\|_2V_N$ where $V_N$ is 
independent of $Z$, so that $\mathbf{E}\|Z\|^2_{(\nu N)} =
N\,\mathbf{E}\|V_N\|^2_{(\nu N)}$.
\end{proof}

We can now complete the proof of Proposition \ref{prop:unitsphere}.

\begin{proof}[Proof of Proposition \ref{prop:unitsphere}]
It is shown in \cite[Theorem 2.10]{eigenvectors-survey} that
$$
	\mathbf{P}\big[\big|\|V_N\|_{(\nu N)} - 
		\mathbf{E}\|V_N\|_{(\nu N)}\big|
	>t\big] \le C e^{-cNt^2},
$$
where $c,C>0$ are universal constants. 
This implies that $\mathrm{Var}(\|V_N\|_{(\nu N)})=O(\frac{1}{N})$, 
so that $\mathbf{E}\|V_N\|_{(\nu N)} = 
(\mathbf{E}\|V_N\|_{(\nu N)}^2)^{1/2}+o(1)$ as $N\to\infty$. 
Using Lemma \ref{lem:sphorder} and applying the above tail estimate 
again yields
$$
	\begin{aligned}
	&\mathbf{P}[V_N\in D_{\nu N,\kappa}]=1-o(1)
	&&\text{when }\textstyle{C'\nu\log\frac{e}{\nu}\le \kappa^2},\\
	&\mathbf{P}[V_N\in D_{\nu N,\kappa}]=o(1) 
	&&\text{when }\textstyle{c'\nu\log\frac{e}{\nu}\ge\kappa^2}
	\end{aligned}
$$
for some universal constants 
$c',C'>0$. The conclusion follows readily.
\end{proof}

\subsection*{Acknowledgments}

LS was supported in part by NSF Graduate Research Fellowship under Grant 
No. DGE-2039656, and Simons Foundation grant 563916, SM.
RvH was supported in part by the NSF grants DMS-1856221 and 
DMS-2054565. We thank the referee for helpful remarks that improved
the presentation.

\subsection*{Data Availibility}

Data sharing is not applicable to this article as no datasets were 
generated or analysed during
the current study.

\subsection*{Conflict of interest}

The authors have no conflicts of interest.

\bibliographystyle{abbrv}
\bibliography{./rmat_ref.bib}

\end{document}